\documentclass{kms-j}

\issueinfo{}
  {}
  {}
  {}
\pagespan{1}{}
\copyrightinfo{}
  {The Korean Mathematical Society}

\theoremstyle{plain}
\usepackage{amsfonts}
\usepackage{mathrsfs}
\newcommand{\dif}{\mathrm{d}}

\newcommand{\e}{\mathrm{e}}
\newcommand{\tr}{\mathrm{tr}}

\newcommand{\R}{\mathrm{R}}

\newtheorem{theorem}{Theorem}[section]
\newtheorem{lemma}[theorem]{Lemma}

\newtheorem{corollary}[theorem]{Corollary}
\newtheorem{proposition}[theorem]{Proposition}
\newtheorem{assumption}[theorem]{Assumptions}

\theoremstyle{definition}

\theoremstyle{remark}
\newtheorem{remark}[theorem]{Remark}

\theoremstyle{example}
\newtheorem{example}[theorem]{Example}
\numberwithin{equation}{section}

\begin{document}

\title[Curvature flows depending on mean curvature] {Contracting convex hypersurfaces by functions of the mean curvature}

\author{Shunzi Guo}
\address{Department of Mathematics and statistics Science\\Minnan
Normal University\\Zhangzhou 363000, P.R.China \\
and
School of Mathematics \\Sichuan University\\Chengdu 610065, P.R.China}
\email{guoshunzi@yeah.net}

\thanks{}

\subjclass[2010]{Primary 53C44, 35K55, 58J35, 35B40}
\keywords{curvature flow, convex hypersurface, fully nonlinear.}

\begin{abstract}
This paper concerns the evolution of a closed convex hypersurface in ${\mathbb{R}}^{n+1}$,
in direction of its inner unit normal vector,
where the speed is given by a smooth function depending only on the mean curvature,
and satisfies some further restrictions, without requiring homogeneity.
It is shown that the flow exists on
a finite maximal interval, convexity is preserved and
the hypersurfaces shrink down to a single point as the final time is
approached.
This generalises the corresponding result of Schulze \cite{Sch05} for the positive power
mean curvature flow to a much larger possible class of flows by the functions depending only on the mean curvature.
\end{abstract}

\maketitle

\section{Introduction and main result}

Let $M^n$ be a smooth, compact oriented manifold of
dimension $n (\geq 2)$ without boundary, and
$X_{0}:M^{n}\rightarrow {\mathbb{R}}^{n+1}$ a smooth immersion of $M^n$ into the euclidean space.
Consider a one-parameter family of
smooth immersions: $\mathrm{X}_{t}:M^{n}\rightarrow {\mathbb{R}}^{n+1}$, evolving according to
\begin{equation}
\left\{
\begin{array}{ll}
\frac{\partial}{\partial t}\mathrm{X}\left(p,t\right)= -\Phi(H)\left(p,t\right)\cdot \nu\left(p,t\right),& p\in M^{n},\\[2ex]
\mathrm{X}(\cdot,0)=\mathrm{X_{0}}(\cdot),
\end{array}\right.\label{eqation:1.1}
\end{equation}
where $\nu\left(p,t\right)$ is the outer unit normal to
$M_{t}$ at $\mathrm{X}\left(p,t\right)$ in tangent space $TN^{n+1}$,
$\Phi$ is a smooth supplementary function defined on an open subset in ${\mathbb{R}}$ and satisfying $\Phi' > 0$,
and $H\left(p,t\right)$ the trace of Weingarten map
$\mathscr{W}_{-\nu}\left(p,t\right)=-\mathscr{W}_{\nu}\left(p,t\right)$
on tangent space $TM^{n}$ induced by $\mathrm{X}_{t}$. Throughout
the paper, we will call such a flow {$\Phi(H)$-flow}.

For $\Phi(H) =H$, we obtain the well-known mean curvature flow,
Huisken \cite{Hui84} showed that in the Euclidean
space $\mathbb{R}^{n+1}$ any closed convex hypersurface $M_{0}$
evolving by mean curvature flow contracts to a point in finite time,
becoming spherical in shape as the limit is approached. In
\cite{Hui86}, he extended this result to compact hypersurfaces in
general Riemannian manifolds with suitable bounds on curvature.
For $\Phi(H) =-\frac{1}{H}$, we get the inverse mean curvature
flow, which was studied in Euclidean space and hyperbolic space
\cite{Ger90,Ger11} and other Riemmannian spaces,
in particular, Huisken and Ilmanen \cite{HI01} used it to prove the Penrose inequality for
asymptotically flat $3$-manifolds.
For $\Phi(H) ={H}^{\beta}$, this flow becomes the power mean curvature
flow, which has been considered by Schulze in \cite{Sch05} for $M_{0}$ of strictly positive mean curvature hypersuface in the Euclidean space,
he proved that the ${H}^{\beta}$-flow has a unique, smooth solution on a finite time
interval $[0, T)$ and $M_{t}$ converges to a point as $t\rightarrow
T$ if $M_{0}$ is strictly convex for $0<\beta <1$ or $M_{0}$ is
weakly convex for $\beta \geq 1$. Here weakly convex and strictly
convex, resp., are defined as all the eigenvalues of Weingarten map
being positive and nonnegative, resp.. But some counterexamples show
that in general the evolving hypersurfaces may not become spherical
in shape as the limit is approached.
In the previous paper \cite{GLW13},
the author, together with Li and Wu, extended Schulze¡¯s results to $h$-convex hypersurfaces
in the hyperbolic space,
and showed that if the initial hypersurface has mean curvature bounded below,
the positive power mean curvature flow has a unique, smooth solution on a finite time interval,
and converges to a point if the initial hypersurface is strictly $h$-convex for the
case that $0<\beta <1$, or weakly convex for $\beta \geq 1$.
Moreover, for the ${H}^{\beta}$-flow case with $\beta \geq 1$,
it has been found that if the initial hypersurface
$M_{0}$ has the ratio of largest to smallest principal curvatures close enough to 1 at every point,
then the evolving hypersurfaces contract to a round point: This was first shown in the Euclidean space setting
by Schulze \cite{Sch06}, then in the hyperbolic space setting by the author, Li and Wu \cite{GLW16}.

A feature of the results mentioned above is that the speeds for these flows
all depend only on the mean curvature $H$.
However, for other flows with the speeds given by arbitrary functions $\Phi$ depending only on $H$,
the understanding of convergence is far less complete, except in some specific settings
such as closed convex surface expanding (for example, see \cite[Smoczyk]{Smo97}).
There are many difficulties in understanding such flows with arbitrary speed function $\Phi$:
The first difficulty is to choose ``nice" speeds depending only on $H$ which guarantee that
we have suitable inequalities on curvatures  along such flows which
furthermore ensure that local convexity of initial data are preserved.
The second difficulty stems from the greatly increased complexity caused by the presence of the arbitrary function $\Phi$;
for instance, the application of the maximum principle to the evolution equations for geometric quantities
either fail or become more subtle, and deriving the sufficient regularity results of solutions for such flows
become potentially more complicated than that for the usual geometric flow.

The present paper considers a wide class of such flows with the ``nice" speeds $\Phi(H)$, which
satisfy the following additional conditions:
\begin{assumption} \label{conditions on Phi}
Let $\Phi:(0, +\infty)\rightarrow \mathbb{R}$ is a smooth function such that for all
$x\in (0, +\infty)$ we have
 \begin{align*}
\Phi>0, \Phi'>0, \Phi''\geq\frac{-2\Phi'}{x} \quad\text{and}\quad \Phi\Phi''x+\Phi \Phi'-\left(\Phi'\right)^{2}x\geq0.
\end{align*}
\end{assumption}

Then we will show the additional technical conditions on the $\Phi$ which determine
that the general picture of behaviour established in the positive power mean curvature flow  by \cite[Schulze]{Sch05},
remain valid for the $\Phi$-flow case.
For convenience, we define a function $g(x)=\frac{1}{\Phi\left(\frac{1}{x}\right)}$ on $(0, +\infty)$,
and set $G(x)=\int_{0}^{x}g(s)\dif s$.
The main result achieved can be exactly stated by the following theorem.
\begin{theorem}\label{main result for Phi flow}
Assume that $\mathrm{X_{0}}:M^{n}\rightarrow {\mathbb{R}}^{n+1}$ be a
smooth convex immersion and that the smooth function $\Phi:(0,+\infty)\rightarrow {\mathbb{R}}$ is
strictly increasing. Then there exists a unique, smooth solution to the flow \eqref{eqation:1.1} on a finite maximal time
interval $[0, T)$.
Furthermore, if the function $\Phi(x)$ for $x>0$ satisfies the assumptions \ref{conditions on Phi},
then $T$ is between $G\left(\frac{1}{H_{\max}(0)}\right)$ and $ nG\left(\frac{1}{H_{\min}(0)}\right)$.
In particular, in the following two cases that
\begin{itemize}
\item[i)] $M_{0}$ is strictly convex for $\frac{-2\Phi'}{H}\leq \Phi''<0$,
\item[ii)] $M_{0}$ is weakly convex for $\Phi''\geq0$ and $\Phi'\geq\frac{\Phi}{H}$,
\end{itemize}
then the hypersurfaces $M_{t}$ are strictly  convex for all $t >0$
and they contract to a point in ${\mathbb{R}}_{\kappa}^{n+1}$ as $t$
is approached to $T$.
\end{theorem}

\begin{remark}
\begin{itemize}
\item [1.] The positivity on the first order derivative of the function $\Phi$ is essential to ensure
short-time existence like the $H^\beta$-flow case in \cite{Sch05}.
\item [2.] In order to drop ``bad" terms in the evolution equation for geometric quantities
and then apply the maximum principle to show monotonicity of curvature,
the more assumptions on $\Phi$, which is similar as that of \cite[Theorem 1]{Smo97}, are required.
However our hypotheses of $\Phi$ differ from those in \cite{Smo97}
in one important respect: No extra assumption $\left(\frac{\Phi''\Phi'}{x}\right)'\leq 0$ on $\Phi$ is required,
due to focusing on different problems.
This shows that our assumptions \ref{conditions on Phi} are more weaker than those in \cite[Theorem 1]{Smo97}.
\end{itemize}
\end{remark}

There exist many examples of natural flows with the speeds $\Phi$ satisfying the assumptions \ref{conditions on Phi}
which not covered by previous results, for example,
\begin{example}
\begin{itemize}
\item [(i)] $\Phi(x)=\beta_{1}{x}^{\beta_{2}}+\beta_{3}$ defined  on all $x\in(0, +\infty)$,
such that the constants $\beta_{i}>0$, $i=1,2$, and $\beta_{3}\geq 0$.
Obviously, these include $\Phi(H)=H$ (the mean curvature) and $\Phi(H)=H^{\beta}$ (the positive powers of the mean curvature).
In particular, the case that $\beta_{2}=1$ and $\beta_{3}=0$, i.e. $\Phi(x)=\beta_{1}{x}$ defined on all $x\in(0, +\infty)$,
corresponds the case $ii)$ in Theorem \ref{main result for Phi flow}.
And the case $\beta_{2}>1$,
in this situation where $\Phi(x)$ defined on
$\left(\left(\frac{\beta_{3}}{\beta_{1}(\beta_{2}-1)}\right)^{\frac{1}{\beta_{2}}}, +\infty\right)$
also corresponds the case $ii)$ in Theorem \ref{main result for Phi flow}.
\item [(ii)] $\Phi(x)=\beta_{1}{\sinh^{\beta_{2}}(x)}+\beta_{3}$:
 \begin{itemize}
\item [(1)]  on all $x\in(0, +\infty)$ and for the constants $\beta_{i}>0 $, $i=1,3$, and $\beta_{2}\geq 1$.
\item [(2)]  on all $x\in\left(\ln\left(\beta_{2}^{-\frac{1}{2}}+\sqrt{\beta_{2}^{-1}-1}\right), +\infty\right)$
and for the constants $\beta_{i}>0 $, $i=1,3$, and $0\leq\beta_{2}< 1$.
 \end{itemize}
 In particular, the case that $\beta_{2}>1$ and $\beta_{3}=0$, i.e. $\Phi(x)=\beta_{1}{\sinh^{\beta_{2}}(x)}$ defined on all $x\in(0, +\infty)$,
corresponds the case $ii)$ in Theorem \ref{main result for Phi flow}.
\item [(iii)] $\Phi(x)=\beta_{1}{e}^{\beta_{2}x}+\beta_{3}$
on all $x\in(0, +\infty)$ and for the constants $\beta_{i}>0 $, $i=2,3$, and $\beta_{1}\geq 1$.
\end{itemize}
\end{example}

\begin{remark}
\begin{itemize}
\item [(i)] All of the above examples can be used in Theorem \ref{main result for Phi flow}.
Note that of these, relatively few are covered by the previously results,
for example, for the flows with the speed functions $\Phi(x)=\beta_{1}{\sinh^{\beta_{2}}(x)}+\beta_{3}$
the understanding of behaviors is far less developed.
\item [(ii)] Furthermore, observe that the main result of Theorem \ref{main result for Phi flow} does not require any homogeneity condition,
as in \cite{And94, Andr94}, et al.. Nevertheless, our results are a significant extension of those in \cite{Sch05} in this direction.
\item [(iii)] As mentioned in \cite[Smoczyk]{Smo97},
one can easily check that the function $\Phi(x)=\ln x$ satisfies almost all conditions of assumptions \ref{conditions on Phi},
only except in the latest condition such that $\Phi\Phi''x+\Phi \Phi'-\left(\Phi'\right)^{2}x<0$.
 \end{itemize}
\end{remark}

We use mainly the methods used in \cite{Sch05} to prove the above theorem,
but with technical tricks for choosing the right functions $\Phi$ to get estimates,
which take control of the complications due to the presence of the arbitrary function $\Phi$.
The organization of the paper is as
follows: Section 2 introduces the notation for the paper and
summarize preliminary results employed in the rest of the paper.
Section 3 contains details of short-time existence and uniqueness of
solutions and the evolution equations of some geometric
quantities, this requires only minor modifications of the power mean curvature flow case
due to the more general $\Phi(H)$-flow case.
Section 4 shows the lower and above bounds on the maximal time, and
establishes the higher-order regularity which give rise to
the long time existence for solutions of the flow \eqref{eqation:1.1}.
Using these, section 5 deduces that solutions of the flow
\eqref{eqation:1.1} remain convex as long as it exists and proves that these hypersurfaces shrink down to a single
point in $\mathbb{R}^{n+1}$ as the final time is
approached.

\section{Notation and preliminary results}
From now on, use the same notation as in \cite{Cab,Hui84,Sch05} in
local coordinates $\{x^{i}\}$, $1\leq i \leq n$, near $p\in M^{n}$
and $\{y^{\alpha}\}$, $0 \leq \alpha, \beta \leq n$, near $F(p)\in
\mathbb{R}^{n+1}$. Denote by a bar all quantities on
$\mathbb{R}^{n+1}$, for example by
$\bar{g}=\{\bar{g}_{\alpha \beta}\}$ the metric, by
$\bar{g}^{-1}=\{\bar{g}^{\alpha \beta}\}$ the inverse of the metric,
by $\bar{y}=\{\bar{y}^{\alpha}\}$ coordinates, by $\bar{\nabla}$ the
covariant derivative, by $\bar{\Delta}$ the rough Laplacian, and by
$\bar{\R}=\{\bar{\R}_{\alpha\beta\gamma\delta}\}$ the Riemann
curvature tensor. Components are sometimes taken with respect to the
tangent vector fields $\partial_{\alpha}(=\frac{\partial}{\partial
y^{\alpha}})$ associated with a local coordinate $\{y^{\alpha}\}$
and sometimes with respect to a moving orthonormal frame
$e_{\alpha}$, where $\bar{g}(e_{\alpha},
e_{\beta})=\delta_{\alpha\beta}$. The corresponding geometric
quantities on $M^{n}$ will be denoted by $g$ the induced metric, by
$g^{-1}, \nabla, \Delta, \R,
\partial_{i}$ and $e_{i}$. Then further important quantities are
the second fundamental form $A(p)=\{h_{ij}\}$ and the Weingarten map
$\mathscr{W}=\{g^{ik}h_{kj}\}=\{h^{i}_{j}\}$ as a symmetric operator
and a self-adjoint operator respectively. The eigenvalues
$\lambda_{1}(p)\leq \cdots \leq \lambda_{n}(p)$ of $\mathscr{W}$ are
called the principal curvatures of $\mathrm{X(M^{n})}$ at
$\mathrm{X(p)}$. The mean curvature is given by
\[H:=\tr_{g}{\mathscr{W}}=h^{i}_{i}=\sum_{i=1}^{n}\lambda_{i},\]
the total curvature by
$$\bigl|A\bigr|^{2}:=\tr_{g}({\mathscr{W}^{t}\mathscr{W}})=h^{i}_{j}h^{j}_{i}=h^{ij}h_{ij}=\sum_{i=1}^{n}\lambda^{2}_{i},$$
and $\mbox{Gau\ss}$-Kronecker curvature by
$$K:=\det(\mathscr{W})=\det\{h^{i}_{j}\}=\frac{\det\{h_{ij}\}}{\det\{g_{ij}\}}=\prod_{i=1}^{n}\lambda_{i}.$$
More generally, the mixed mean curvatures $E_{r}, 1\leq r\leq n,$
are given by the elementary symmetric functions of the $\lambda_{i}$
$$E_{r}(\lambda)=\sum\limits_{1 \leq
i_{1}\leq  \cdots \leq i_{r} \leq n}\lambda_{i_{1}}\cdots
\lambda_{i_{r}} =\frac{1}{r!}\sum\limits_{
i_{1},\ldots,i_{r}}\lambda_{i_{1}}\cdots \lambda_{i_{r}},\ \ {\hbox
{for}}\  \lambda = (\lambda_{1},\ldots,\lambda_{n})\in
\mathbb{R}^{n},$$ and their quotients are
$$Q_{r}(\lambda)=\frac{E_{r}(\lambda)}{E_{r-1}(\lambda)},\ \ {\hbox
{for}}\  \lambda \in \Gamma_{r-1},$$ where $E_{0}\equiv 1,$ and
$E_{l}\equiv 0,$ if $r>n$, $\Gamma_{r}:=\{\lambda \in
\mathbb{R}^{n}| E_{i} > 0, i= 1, \ldots, r\}.$ Denote the sum of all
terms in $E_{r}(\lambda)$ not containing the factor $\lambda_{i}$ by
$E_{r;i}(\lambda)$. Then the following identities for $E_{r}$ and
the properties on the quotients $Q_{r}$ were proved by Huisken and
Sinestrari in \cite{Hui99}.

\begin{lemma}
For any $r \in \{1, \ldots, n\},$ $i \in \{1, \ldots, n\},$ and
$\lambda \in \mathbb{R}^{n}$,
\begin{align*}
\frac{\partial E_{r+1}}{\partial \lambda_{i}}(\lambda)&=E_{r;i}(\lambda),\\
E_{r+1}(\lambda)&= E_{r+1;i}(\lambda)+\lambda_{i}E_{r;i}(\lambda),\\
\sum_{i=1}^{n}E_{r;i}(\lambda)&=(n-r)E_{r}(\lambda),\\
\sum_{i=1}^{n}\lambda_{i}E_{r;i}(\lambda)&=(r+1)E_{r+1}(\lambda),\\
\sum_{i=1}^{n}\lambda^{2}_{i}E_{r;i}(\lambda)&=E_{1}(\lambda)E_{r+1}(\lambda)-(r+2)E_{r+2}(\lambda).
\end{align*}
\end{lemma}

\begin{lemma}
i) $Q_{r+1}$ is concave on $\Gamma_{r}$ for $r \in \{0, \ldots,
n-1\}$,\\
ii) $\frac{\partial Q_{r}}{\partial \lambda_{i}}(\lambda)> 0$ on
$\Gamma_{r}$ for $i \in \{1, \ldots, n-1\}$ and $r \in \{2, \ldots,
n-1\}$.
\end{lemma}

In graphical coordinates,
one can adopt a local graph representation for a convex hypersurface given by a height function $u$.
For future reference, it is useful to recall here some basic formulae in a graphical representation.
First we see that
\begin{equation*} 
X(p,t)=\left(x(p,t), u(x(p,t),t)\right).
\end{equation*}
So the metric and its inverse are given by
\begin{equation*} 
g_{ij} = \delta_{ij} + D_i u \, D_j u, \quad  \quad \quad g^{ij} = \delta^{ij} - \frac{D^i u \, D^j u}{1 + |D u|^2},
\end{equation*}
where $D_i$ denote the derivatives with respect to these local coordinates,
respectively.
The outward unit normal vector of
$M_{ t}$ can be expressed as
\begin{equation}\label{param outward unit normal2}
\nu=\frac{1}{\left|\xi\right|}\Big( -D u,1\Big)
\end{equation}
with
\begin{equation}\label{normal vector for graph}
\big|\xi\big|=\sqrt{1+\left|D u\right|^{2}}.
\end{equation}
The second fundamental form can be expressed as
\begin{equation*}
h_{ij} =\frac{D_{ij} u}{(1 + |Du|^2)^{1/2}},
\end{equation*}
 and
\begin{equation*}
h_{j}^{i} =\left( \delta^{ik} - \frac{D^i u \, D^k u}{1 + |D u|^2}\right) \frac{D_{kj} u}{(1 + |Du|^2)^{1/2}},
\end{equation*}
Then we obtain
    \begin{equation}\label{H on graph}
      H =g^{ij}h_{ij}      =\left(\delta^{ij}-\frac{u^iu^j}{1+|Du|^2}\right)
      \frac{u_{ij}}{\sqrt{1+|Du|^2}}
          \end{equation}

In addition, the Christoffel symbols have the expression:
\begin{equation}
\label{invg_u}
\Gamma_{ij}^k = \left(\delta^{kl} - \frac{D^k u \, D^l u}{1 + |D u|^2}\right) D_{ij} u \, D_l u.
\end{equation}

\section{Short time existence and evolution equations}
This section first consider short time existence for the initial
value problem \eqref{eqation:1.1}.
In order to obtain these results, it suffices to demand that
$\Phi'$ is strictly positive.
\begin{theorem}
Assume that $\mathrm{X_{0}}:M^{n}\rightarrow {\mathbb{R}}^{n+1}$ be a
smooth immersion and that the smooth function $\Phi:[0,+\infty)\rightarrow {\mathbb{R}}$
is strictly monotone increasing.
Then there exists a unique smooth solution
$\mathrm{X}_{t}$ of problem \eqref{eqation:1.1}, defined on some time
interval $[0, T )$, with $T > 0$.
\end{theorem}
\begin{proof}
In fact, if $f$ is any symmetric function of the curvatures
$\lambda_{i}, i\in \{1, \ldots, n\},$ it is well known (see e.g.
Theorem $3.1$ of \cite{HP99}) that a flow of the form
\begin{equation*}
\frac{\partial}{\partial t}\mathrm{X}\left(p,t\right)=
-f\left(p,t\right)\cdot\nu\left(p,t\right)
\end{equation*}
is parabolic on a given hypersurface with the condition
$\frac{\partial f}{\partial \lambda_{i}}>0$ for all $i$ holds
everywhere. Then, given any initial immersion $\mathrm{X_{0}}$
satisfying the parabolicity assumption, standard techniques ensure
the local existence and uniqueness of a solution to \eqref{eqation:1.1}
with initial value $\mathrm{X_{0}}$. In our case $f=\Phi(H)$ and
the condition reads
\begin{equation*}
\frac{\partial \Phi(H)}{\partial \lambda_{i}}= \Phi'\frac{\partial {H}}{\partial \lambda_{i}}
=\Phi'>0,
\end{equation*}
which is satisfied the condition of Theorem $3.1$ of \cite{HP99}.
\end{proof}

By a direct calculation as in \cite{Hui84}, or \cite{And94},
the following evolution equations of geometric quantities under the flow \eqref{eqation:1.1} can be easily obtained.
\begin{theorem}
On any solution $M_{t}$ of \eqref{eqation:1.1} the following hold:
\begin{align}
\partial_{t}g_{ij}&=-2\Phi(H) h_{ij},\label{metic evolution for Phih flow}\\
\partial_{t}\nu&=\nabla \Phi(H),\notag\\
\partial_{t}(\dif \mu_{t})&=-H\Phi\dif \mu_{t},\notag\\
\partial_{t}h_{ij}&=\Delta_{\dot{\Phi}} h_{ij}
+\Phi''\nabla_{i}H\nabla_{j}H-(\Phi' H+\Phi)h_{i}^{k}h_{kj} +\bigl|A\bigr|^{2}\Phi'h_{ij}, \label{evolution for h}\\
\partial_{t}h_{i}^{j}&=\Delta_{\dot{\Phi}} h_{i}^{j}\label{the evolution for weigarten}
+\Phi''\nabla_{i}H\nabla^{j}H-(\Phi' H-\Phi)h_{i}^{k}h^{j}_{k},\\
\partial_{t}H&= \Delta_{\dot{\Phi}} H +\Phi''|\nabla H|^{2}
+\bigl|A\bigr|^{2}\Phi,\label{the evolution for H}\\
\partial_{t} \langle\mathrm{X},  \nu\rangle &= \Delta_{\dot \Phi} \langle\mathrm{X},  \nu\rangle
+ \bigl|A\bigr|^{2}\Phi'\langle\mathrm{X},  \nu\rangle- (\Phi' H+\Phi).\label{evolution for support function}
\end{align}
\end{theorem}

Furthermore, the quotients ${Q}_{r}({\lambda})$ satisfy
the following evolution equation which is an extension of
\cite[ Lemma\,2.4]{Sch05} to hypersurfaces of \eqref{eqation:1.1} in
${\mathbb{R}}^{n+1}$ :
\begin{lemma}\label{evolution for the quotient of Qr}
Suppose $\Phi$ satisfies that $\Phi''\geq 0$ and $ \Phi'{H}\geq {\Phi}$.
Let $\mathrm{X}:M^{n}\times[0,T)\rightarrow
{\mathbb{R}}^{n+1}$ be a $\Phi(H)$-flow with
\[
{E}_{r-1}\left(p,t\right)>0, \quad
{E}_{r}\left(p,t\right)\geq 0 \quad \text{for all}\
\left(p,t\right)\in M^{n}\times[0,T).
\]
Then
\begin{align}
\partial_{t}{Q}_{r}\geq\Phi'\Delta{Q}_{r}+ \left[\Phi'\bigl|{A}\bigr|^{2}-r (\Phi' H-\Phi){Q}_{r}\right]{Q}_{r}.
\end{align}
\end{lemma}

\begin{proof}
As in \cite{Sch05}, from the evolving equation \eqref{evolution for h} of
${h}_{i}^{j}$, using
\[\partial_{t}{Q}_{r}=\frac{\partial{Q}_{r}}{\partial{h}_{i}^{j}}\left(\partial_{t}{h}_{i}^{j}\right)
\quad \text{and} \quad
\Delta{Q}_{r}=\frac{\partial{Q}_{r}}{\partial{h}_{i}^{j}}\Delta{h}_{i}^{j}
+\frac{\partial^{2}{Q}_{r}}{\partial{h}_{i}^{j}\partial{h}_{p}^{q}}\nabla^{k}
{{h}}_{i}^{j}\nabla_{k}{{h}}_{p}^{q}
\]
it is easy to calculate the derivative of ${Q}_{r}$:
\begin{align*}
\partial_{t}{Q}_{r}&=\Phi'\Delta {Q}_{r}
-\Phi'\frac{\partial^{2}{Q}_{r}}{\partial{h}_{i}^{j}\partial{h}_{p}^{q}}
\nabla^{k}{{h}}_{i}^{j}\nabla_{k}{{h}}_{p}^{q} +\Phi''\frac{\partial{Q}_{r}}{\partial{h}_{i}^{j}}
\nabla_{i}{H}\nabla^{j}{H}
\\
 &\quad-(\Phi' H-\Phi)\frac{\partial{Q}_{r}}{\partial{h}_{i}^{j}}{h}_{i}^{k}{h}^{j}_{k}+ \Phi'\bigl|{A}\bigr|^{2}\frac{\partial{Q}_{r}}{\partial{h}_{i}^{j}}h^{j}_{i}.
\end{align*}
Choosing a frame $\{e_{i}\}$ which diagonalises
$\mathscr{{W}}$, the fifth and sixth term appearing here can
be simplified using the following simple calculation with the aid of
Lemma 2.1:
\begin{align*}
\frac{\partial{Q}_{r}}{\partial{h}_{i}^{j}}{h}_{i}^{k}{h}^{j}_{k}
&=\frac{\partial{Q}_{r}}{\partial{h}_{i}^{j}}{{h}}_{i}^{k}{h}^{j}_{k}=\sum_{i=1}^{n}\frac{\partial{Q}_{r}}{\partial{\lambda}_{i}}{\lambda}^{2}_{i}
\\
&=\frac{1}{E^{2}_{r-1}}\left(E_{r-1}\sum_{i=1}^{n}E_{r-1,i}\lambda^{2}_{i}
-E_{r}\sum_{i=1}^{n}E_{r-2,i}\lambda^{2}_{i}\right)\\
&=-(r+1)\frac{E_{r}}{E_{r-1}}+r{Q}_{r}
\end{align*}
and
\begin{align*}
\frac{\partial{Q}_{r}}{\partial{h}_{i}^{j}}{h}^{j}_{i}
&=\sum_{i=1}^{n}\frac{\partial{Q}_{r}}{\partial\lambda_{i}}\lambda_{i}\\
&=\frac{1}{E^{2}_{r-1}}\left(E_{r-1}\sum_{i=1}^{n}E_{r-1,i}\lambda_{i}
-E_{r}\sum_{i=1}^{n}E_{r-2,i}\lambda_{i}\right)\\
&={Q}_{r}.
\end{align*}
In view of the Lemma\,2.2, the second, the third and the last term
in the right hand side of the evolution equation of
$Q_{r}$ are positive by monotonicity and concavity of the
${Q}_{r}$. So the desired inequality can be obtained with the
hypotheses.
\end{proof}

If the hypersurfaces $M_{t}$ are strictly convex, consider the
inverse $\mathscr{{W}}_{p}^{-1}$ of $\mathscr{{W}}_{p}$
at a given point $p\in M^{n}$, set
$\mathscr{{W}}_{p}^{-1}=\{{b}_{i}^{j}\}$, where
${b}_{i}^{j}$ is given by
${b}_{i}^{k}{h}_{k}^{j}={\delta}_{i}^{j}$. The evolution
equation of ${b}_{i}^{j}$ is similar to the $H^{\beta}$-flow case:
\begin{lemma}\label{priciple radii evolution}
For $\Phi'>0$, $ \Phi''\geq -\frac{2\Phi'}{H}$, let $M_{t}$ be a $\Phi(H)$-flow of strictly
convex hypersurfaces in ${\mathbb{R}}^{n+1}$. Then
\begin{align}\label{the evolution for pricipal radii}
\partial_{t}{b}_{i}^{j}
&=\Delta_{\dot{\Phi}} {b}_{i}^{j} -2\Phi'\left(\nabla^{k}{{b}}_{i}^{p}\right){h}_{p}^{q}\left(\nabla_{k}{{b}}_{q}^{j}\right)
- \Phi''\left({{b}}_{i}^{p}\nabla_{p}{H}\right)\left(\nabla^{q}{H}{{b}}_{q}^{j}\right)\\
 &\quad+(\Phi' H-\Phi){\delta}_{i}^{j}
-\Phi'\bigl|{A}\bigr|^{2}{b}_{i}^{j}\notag
\\
&\leq \Delta_{\dot{\Phi}} {b}_{i}^{j}
+(\Phi' H-\Phi){\delta}_{i}^{j}
-\Phi'\bigl|{A}\bigr|^{2}{b}_{i}^{j}.\notag
\end{align}
\end{lemma}
\begin{proof}
Compute from ${b}_{i}^{k}{h}_{k}^{j}={\delta}_{i}^{j}$
\[
\partial_{t}{b}_{i}^{j}=-{{b}}_{i}^{p}\left(\partial_{t}{h}_{p}^{q}\right){{b}}_{j}^{q}
\]
and
\[
\nabla_{k}{b}_{i}^{j}=-{{b}}_{i}^{p}\left(\nabla_{k}{h}_{p}^{q}\right){{b}}_{j}^{q}
\]
which implies
\[
\Delta
{b}_{i}^{j}=-{{b}}_{i}^{p}\left(\Delta{h}_{p}^{q}\right){{b}}_{j}^{q}
+2\nabla^{k}{{b}}_{i}^{p}{h}_{p}^{q}\nabla_{k}{{b}}_{j}^{q}.
\]
Together with equation \eqref{the evolution for weigarten}, this gives the equality.

{\bf Case 1. }
For $\Phi''\geq 0$, the inequality follows immediately.

{\bf Case 2. }
For $-\frac{2\Phi'}{H}\leq \Phi''< 0$, the two gradient terms on the  right hand side of the equality in
Lemma \ref{priciple radii evolution} have the desired sign, we have to work a bit more.
Note that
\begin{equation}\label{inverse curvature relation}
- \Phi''\left({{b}}_{i}^{p}\nabla_{p}{H}\right)\left(\nabla^{q}{H}{{b}}_{q}^{j}\right)
= -\frac{\partial^{2}\Phi}{\partial {h}_{l}^{k}\partial{h}_{n}^{m}}
{b}_{i}^{p}\nabla_{p}{{h}_{l}^{k}}\nabla^{q}{h}_{n}^{m}{{b}}_{q}^{j}.
\end{equation}
As in
(\cite{Sch05}, Lemma 2.5), note that $H(\lambda)=Q_{n}(\theta)$,
where the $\theta_{i}$ are the principle radii, i.e., $\theta_{i}=\frac{1}{\lambda_{i}}$.
For general functions $f, g$
satisfying $f({{h}}_{i}^{j})=1/g({{b}}_{i}^{j})$ one can
compute that
\begin{equation}
\frac{\partial^{2}f}{\partial {h}_{i}^{j}\partial
{h}_{p}^{q}} =\frac{2}{f}\frac{\partial f}{\partial
{h}_{i}^{j}}\frac{\partial f}{\partial{h}_{p}^{q}}
-f^{2}\frac{\partial^{2}f}{\partial {h}_{m}^{n}\partial
{h}_{k}^{l}}{{b}}^{ni}{{b}}_{mj}{{b}}^{kp}{{b}}_{lq}
-\frac{\partial f}{\partial{h}^{jq}}{h}^{ip}
-\frac{\partial
f}{\partial{h}_{ip}}{h}_{jq}.\label{a first deirivate formula}
\end{equation}
By the chain rule
\begin{equation*}
\frac{\partial\Phi}{\partial{h}_{p}^{q}}({\lambda})
=\Phi'\frac{\partial{H}}{\partial{h}_{p}^{q}}
=\Phi'{\delta_{q}^{p}}
\end{equation*}
and
\begin{equation}\label{a general deirvate formular}
\frac{\partial^{2}\Phi}{\partial
{h}_{m}^{n}\partial {h}_{l}^{k}}
=\Phi''{\delta_{n}^{m}}{\delta_{l}^{k}}+\Phi'\frac{\partial^{2}H}{\partial {h}_{m}^{n}\partial
{h}_{l}^{k}}.
\end{equation}
From \eqref{a first deirivate formula} (with $f=H$) and \eqref{a general deirvate formular}, it
follows
\begin{align}
-\frac{\partial^{2}H}{\partial
{h}_{n}^{m}\partial {h}_{l}^{k}}{b}_{i}^{p}\nabla_{p}{{h}_{l}^{k}}\nabla^{q}{h}_{n}^{m}{{b}}_{q}^{j}
&=-\frac{2}{H}{b}_{i}^{p}\nabla_{p}H\nabla^{q}H{{b}}_{q}^{j}
+H^{2}\frac{\partial^{2}{Q}_{n}}{\partial{b}_{r}^{s}\partial
{b}_{t}^{u}}{b}_{i}^{p}\nabla_{p}{{h}_{s}^{r}}\nabla^{q}{h}_{t}^{u}{{b}}_{q}^{j}\notag\\
&\quad
 +2\nabla_{k}{b}_{i}^{p}{h}_{q}^{p}\nabla^{k}{h}_{q}^{j},
\label{second deirivate formula}
\end{align}
where the Codazzi equation has been used.
Now by identities \eqref{second deirivate formula}, \eqref{inverse curvature relation} and \eqref{a general deirvate formular}
one can write \eqref{the evolution for pricipal radii} as
\begin{align*}
\partial_{t}{b}_{i}^{j}
&=\Delta_{\dot{\Phi}} {b}_{i}^{j}
-2\Phi'\nabla_{k}{b}_{i}^{p}{h}_{q}^{p}\nabla^{k}{h}_{q}^{j}- \Phi''\left({{b}}_{i}^{p}\nabla_{p}{H}\right)\left(\nabla^{q}{H}{{b}}_{q}^{j}\right)
-\frac{2 \Phi'}{H}{b}_{i}^{p}\nabla_{p}H\nabla^{q}H{{b}}_{q}^{j}
\\
 &\quad+ \Phi'H^{2}\frac{\partial^{2}{Q}_{n}}{\partial{b}_{r}^{s}\partial
{b}_{t}^{u}}{b}_{i}^{p}\nabla_{p}{{h}_{s}^{r}}\nabla^{q}{h}_{t}^{u}{{b}}_{q}^{j}
+2\Phi'\nabla_{k}{b}_{i}^{p}{h}_{q}^{p}\nabla^{k}{h}_{q}^{j}+(\Phi' H-\Phi){\delta}_{i}^{j}
-\Phi'\bigl|{A}\bigr|^{2}{b}_{i}^{j}\notag
\\
&=\Delta_{\dot{\Phi}} {b}_{i}^{j}
- \left(\Phi''+\frac{2 \Phi'}{H}\right)
{b}_{i}^{p}\nabla_{p}H\nabla^{q}H{{b}}_{q}^{j}+ \Phi'H^{2}\frac{\partial^{2}{Q}_{n}}{\partial{b}_{r}^{s}\partial
{b}_{t}^{u}}{b}_{i}^{p}\nabla_{p}{{h}_{s}^{r}}\nabla^{q}{h}_{t}^{u}{{b}}_{q}^{j}
\\
 &\quad
+(\Phi' H-\Phi){\delta}_{i}^{j}
-\Phi'\bigl|{A}\bigr|^{2}{b}_{i}^{j}.\notag
\end{align*}
 Using the concavity of ${Q}_{n}(\theta)$ and the assumption $-\frac{2\Phi'}{H}\leq \Phi''\leq 0$,
 it follows that
\begin{align*}
\partial_{t}{b}_{i}^{j}
\leq \Delta_{\dot{\Phi}} {b}_{i}^{j}
+(\Phi' H-\Phi){\delta}_{i}^{j}
-\Phi'\bigl|{A}\bigr|^{2}{b}_{i}^{j}.\notag
\end{align*}.
\end{proof}

\section{The long time existence}
The third section has shown that the equation \eqref{eqation:1.1} has a
(unique) smooth solution on a short time if the initial hypersurface
in $\mathbb{R}^{n+1}$ is  convex. This section consider
the long time behavior of  \eqref{eqation:1.1} and establish the
existence of a solution on a finite maximal interval.

As a first step the maximum principle applied to the evolution
equation of ${H}$ guarantee that the minimum $H_{\min}$ of $H$
is increasing under the flow \eqref{eqation:1.1} which ensures the uniform parabolicity of our equation.
\begin{proposition}\label{minH is increasing}
Under the assumptions of Main Theorem\,1.2,
\[
H_{\min}(t)\geq \frac{1}{G^{-1}\left(G\left(\frac{1}{H_{\min}(0)}
\right)-\frac{t}{n}\right)}
\]
which gives an upper bound on the maximal existence time $T$:
\[
T\leq nG\left(\frac{1}{H_{\min}(0)}\right).
\]
\end{proposition}
\begin{proof}
A direct calculation using $\bigl|A\bigr|^{2}\geq\frac{1}{n}H^{2}$
and the evolution equation \eqref{the evolution for H} of
$ H$ gives
\[
\partial_{t} H_{\min}\geq
\frac{1}{n}\Phi(H_{\min})H_{\min}^{2}.
\]
Now let $\phi$ be the solution of the ODE
\[
\left\{
\begin{array}{ll}
\frac{\dif\phi}{\dif t}=\frac{1}{n}\Phi(\phi)\phi^{2},\\[2ex]
\phi(0)= H_{\min}(0),
\end{array}\right.
\]
then by the maximum principle
\[
 H\geq \phi \qquad \text{on} \qquad  0\leq t \leq T.
\]
On the other hand $\phi$ is explicitly given by
\[
\phi(t)=\frac{1}{G^{-1}\left(G\left(\frac{1}{H_{\min}(0)}
\right)-\frac{t}{n}\right)},
\]
which implies
\[
 H_{\min}(t)\geq\frac{1}{G^{-1}\left(G\left(\frac{1}{H_{\min}(0)}
\right)-\frac{t}{n}\right)}.
\]
Thus,
 \[ H_{\min}(t)\rightarrow\infty \qquad \text{ as} \qquad
G\left(\frac{1}{H_{\min}(0)}
\right)-\frac{t}{n}\rightarrow
0+,
\]
which proves Proposition \ref{minH is increasing}.
\end{proof}

\begin{theorem}\label{the long time existence for PhiH}
Let $[0, T)$ be the maximal existence interval of the flow
\eqref{eqation:1.1} $M_{t}$ with $\Phi'>0$ on
$[\delta_{0},+\infty)$. Then
\[
T\leq nG\left(\frac{1}{H_{\min}(0)}\right).
\]
Moreover, $\max_{M_{t}}\bigl|A\bigr|^{2}\rightarrow +\infty$ as
$t\rightarrow T$.
\end{theorem}
\begin{proof}
The estimates on the maximal time $T$ of existence can be easily
derived from Proposition \ref{minH is increasing}.
To complete the proof of the theorem, assume that
$\bigl|A\bigr|^{2}$ remains bounded on the interval $[0, T)$, and
derive a contradiction. Then the evolution equation \eqref{eqation:1.1}
implies that
\[
\left|\mathrm{X(p,\sigma)}-\mathrm{X(p,\tau)}\right| \leq
\int_{\tau}^{\sigma}\Phi(H)\left(p,t\right)\dif t
\]
for $0\leq \tau \leq \sigma < T$. Since $H$ is bounded from the bound for $\bigl|A\bigr|^{2}$
and the $\Phi$ is a smooth increasing function,
$\mathrm{X(\cdot,t)}$ tends to a unique continuous limit
$\mathrm{X(\cdot,T)}$ as $t\rightarrow T$.
For any $t \in [0, T)$, this implies
the uniform $C^{2}$-estimates for these hypersurface. In order to conclude that
$\mathrm{X(\cdot,T)}$ represents a hypersurface $M_{T}$, next under
this assumption and in view of the evolution equation \eqref{metic evolution for Phih flow}
the induced metric $g$ remains comparable to a fix smooth metric
$\tilde{g}$ on $M^{n}$:
\[
\left|\frac{\partial}{\partial
t}\left(\frac{g(u,u)}{\tilde{g}(u,u)}\right)\right|
=\left|\frac{\partial_{t}
g(u,u)}{g(u,u)}\frac{g(u,u)}{\tilde{g}(u,u)}\right| \leq
2|\Phi(H)||A|_{g}\frac{g(u,u)}{\tilde{g}(u,u)},
\]
for any non-zero vector $u\in T M^{n}$,  so that ratio of lengths is
controlled above and below by exponential functions of time, and
hence since the time interval is bounded, there exists a positive
constant $C$ such that
\[
\frac{1}{C}\tilde{g}\leq g\leq C\tilde{g}.
\]
Then the metrics $g(t)$ for all different times are equivalent, and
they converge as $t\rightarrow T$ uniformly to a positive definite
metric tensor $g(T)$ which is continuous and also equivalent by
following Hamilton's ideas in \cite{Ham82}.

 For $\alpha>0$ the uniform $C^{2,\alpha}$-estimates can be obtained for these
hypersurfaces as follows:  For $-\frac{2\Phi'}{H}\leq \Phi''< 0$,
the speed $\Phi(H)$ is concave in $h_{i}^{j}$ and in this case with  the uniform
$C^{2,\alpha}$-bounds are known in general for operators with
concave (see \cite {Lie}, Theorem 2, Chapter 5.5, or also see \cite
{Kry}). For $\Phi''\geq 0$,  $M^{i}_{t}$ can be locally reparameterized
as graphs given by a height function $u$. From \eqref{eqation:1.1} and
\eqref{param outward unit normal2}, a short computation yields
that height function $u$ satisfies the following parabolic PDE
\begin{equation}\label{graph represent for PhiH}
\partial_{t}u=\Phi(H)\left|\xi\right|,
\end{equation}
 where the mean curvature $H$ and the outward normal vector length $\left|\xi\right|$
 are given by the expressions \eqref{H on graph} and \eqref{normal vector for graph},
 respectively. The function $\Phi(H)$ in the coordinate system
 under consideration is a function of $D^{2}u$ and $Du$. Since
 $H(\cdot,t)$ is larger than $H_{\min}(0)$ and bounded above by our assumption on $\bigl|A\bigr|^{2}$,
 this implies that $\frac{\Phi(H)}{H}$  is also uniformly H\"older
 continuous functions in space and time.
 Using this, we can write equation \eqref{graph represent for PhiH}
 as a linear, strictly parabolic
partial differential equation
\begin{equation}\label{equation:4.4}
\partial_{t}u=a^{ij}D_{i}D_{j}u,
\end{equation}
with coefficients given by
\[
a^{ij}=g^{ij}\frac{\Phi(H)}{H},
\]
in $C^{\alpha}$ in space and time. The interior Schauder estimates by  the general
theory of Krylov and Safonov \cite{Kry}, \cite{Lie} lead to
$C^{2,\alpha}$-estimates. In both cases, i.e. $-\frac{2\Phi'}{H}\leq \Phi''< 0$, and
$ \Phi''> 0$, such a property implies all the higher order estimates by
using standard linearization and bootstrap techniques (see
\cite{Kry}, \cite{Lie}).
 it is enough to imply bounds on all derivatives of $X$. Therefore the
hypersurfaces $M_{t}$ converge to a smooth limit hypersurface
$M_{T}$. Finally, applying the local existence result with initial
data $X(\cdot,t)$, the solution can be continued to a later times,
contradicting the maximality of $T$. This completes the proof of
Theorem 5.1.
\end{proof}

\begin{example}
For the evolution of a  sphere $\mathcal {S}_{0}$ with
a radius $R_{0}$ and the origin point of ${\mathbb{R}}^{n+1}$
its center under the flow \eqref{eqation:1.1}. Since in the sphere case our flow preserves the symmetry,
the equation \eqref{eqation:1.1} reduces to the following ODE for the radius of the spheres
\begin{equation*}
\left\{
\begin{array}{ll}
\frac{\dif R\left(t\right)}{\dif t}= -\Phi\left(\frac{n }{R\left(t\right)}\right),\\[2ex]
R(0)=R_{0}.
\end{array}\right.
\end{equation*}
A straightforward analysis for the existence of solution of the
above ODE implies that the evolving spheres $\mathcal
{S}_{t}$ with radii $R(t)$ contract to the center of the $\mathcal
{S}_{0}$ satisfying
\[
R(t)=n G^{-1}\left(G\left(\frac{R(0)}{n}
\right)-\frac{t}{n}\right),
\]
on a finite maximal existence time $[0,T)$,
where $T$ is given by
\[
T=n \left(G\left(\frac{R(0)}{n}\right)\right).
\]
\end{example}

\section{Preserving convexity}
With the notations of Theorem\,1.2, this section shall show that
convex hypersurface  remains so under the $\Phi(H)$-flow.

To show that convexity of $M_{t}$ is preserved, next consider the
evolution of $\lambda_{\min}:=\min_{M_{t}}\lambda_{i}$ as in
Chap.\,3 of \cite{Hei01}. In order to do so, define a smooth
approximation $\mathcal {A}$ to $\max(x_{1},\ldots,x_{n})$ as
follows: for $\delta
>0$ let
\begin{equation}\label{equation:4.1}
\begin{split}
\mathcal  {A}_{2}(x_{1},x_{2})&=\frac{x_{1}+x_{2}}{2}
+\sqrt{\Big(\frac{x_{1}-x_{2}}{2}\Big)^{2}+\delta^{2}},\\
\mathcal  {A}_{n+1}(x_{1},\ldots,x_{n+1})&=\frac{1}{n+1}
\sum_{i=1}^{n+1}\mathcal  {A}_{2}(x_{i},\mathcal
{A}_{n}(x_{1},\ldots,\hat{x}_{i},\ldots,x_{n+1}),\quad n\geq 2.
\end{split}
\end{equation}
The approximation has the following properties, for a proof see
(\cite{Hei01}, Lemma\,3.3).

\begin{lemma}\label{approxiamtion f}
For $n\geq 2$ and $\delta
>0$,
\begin{itemize}
\item[i)] $\mathcal {A}_{n}(x_{1},\ldots,x_{n})$ is
smooth,monotonically increasing and convex,
\item[ii)] $\max\{x_{1},\ldots,x_{n}\}\leq \mathcal
{A}_{n}(x_{1},\ldots,\ldots,x_{n})\leq
\max\{x_{1},\ldots,x_{n}\}+(n-1)\delta$,
\item[iii)] $\frac{\partial
\mathcal {A}_{n}(x_{1},\ldots,x_{n})}{\partial x_{i}}\leq 1$,
\item[iv)] $\mathcal {A}_{n}(x_{1},\ldots,x_{n})-(n-1)\delta \leq
\sum_{i=1}^{n}\frac{\partial \mathcal
{A}_{n}(x_{1},\ldots,x_{n})}{\partial x_{i}} x_{i}\leq \mathcal
{A}_{n}(x_{1},\ldots,x_{n})$,
\item[v)] $\sum_{i=1}^{n}\frac{\partial \mathcal
{A}_{n}(x_{1},\ldots,x_{n})}{\partial x_{i}}=1$.
\end{itemize}
\end{lemma}
Schulze in \cite{Sch05} proved that the minimal principal curvatures
of the hypersurfaces under the $H^{\beta}$-flow is increasing by
applying the properties of $\mathcal {A}_{n}$, which is also valid
for the $\Phi(H)$-flow .
\begin{lemma} \label{preserving strict convex}
For $\Phi'>0$, $ \Phi''\geq -\frac{2\Phi'}{H}$ let $M_{t}$ be a solution of the $\Phi(H)$-flow \eqref{eqation:1.1}.
Suppose the initial hypersurface
$M_{0}$ is strictly  convex. Then all $M_{t}$ are also strictly
 convex and $\lambda_{\min}(t)$ is monotonically increasing for
$t>0$.
\end{lemma}
\begin{proof} Firstly, note that
Proposition \ref{minH is increasing} ensures that $H$ preserved positivity in time .

{\bf Case 1. } For $\Phi''\geq 0$, using a frame which diagonalises
${\mathscr{W}}$, consider the evolution of
${\lambda}_{\min}(t)$ in the evolution equation \eqref{the evolution for weigarten}
of ${\mathscr{W}}$. Then
\begin{equation}
\begin{split}
\partial_{t}{\lambda}_{\min}\left(p,t\right)
&\geq \Phi'\Delta
{\lambda}_{\min}\left(p,t\right)
-(\Phi' H-\Phi){\lambda}^{2}_{\min}\left(p,t\right)+\bigl|{A}\bigr|^{2}\Phi'{\lambda}_{\min}\left(p,t\right)\\
&=\Phi'\Delta {\lambda}_{\min}\left(p,t\right)
+\Phi{\lambda}^{2}_{\min}\left(p,t\right)+\Phi'\left[\bigl|{A}\bigr|^{2}({\lambda}_{\min}\left(p,t\right))-H{\lambda}^{2}_{\min}\left(p,t\right)\right].
\end{split}
\end{equation}
The part in the square brackets is nonnegative by the estimate $\bigl|{A}\bigr|^{2}\geq
{H}{\lambda}_{\min}$. Then
the maximum principle shows the desired result.

{\bf Case 2. } For $-\frac{2\Phi'}{H}\leq \Phi''< 0$,  observe that the gradient term has the wrong
sign, we have to work a little bit more as in \cite{Sch05}. For a
fixed $\delta> 0$ now choose a smooth approximation
$\mathscr{A}({b}_{i}^{j}):=\mathcal
{A}_{n}(\theta_{1},\ldots,\theta_{n})$ to
$\max(\theta_{1},\ldots,\theta_{n})$ , as defined in \eqref{equation:4.1},
where the $\theta_{i}$ are the eigenvalues of ${b}_{i}^{j}$,
i.e. $\theta_{i}=1/{\lambda}_{i}$. By the chain rule
\[
\partial_{t}\mathscr{A}=\frac{\partial\mathscr{A}}{\partial {b}_{i}^{j}}
\frac{\partial{b}_{i}^{j}}{\partial t}  \quad and \quad \Delta
\mathscr{A} =\frac{\partial\mathscr{A}}{\partial
{b}_{i}^{j}}\Delta {b}_{i}^{j}
+\frac{\partial^{2}\mathscr{A}}{\partial {b}_{r}^{s}\partial
{b}_{t}^{u}}\nabla^{v}{b}_{r}^{s}\nabla_{v}{b}_{t}^{u},
\]
grouping the two identities and applying Lemma {priciple radii evolution} $\mathscr{A}$
satisfies the following evolution inequality:
\begin{align*}
\partial_{t}\mathscr{A}
&\leq \Phi'\Delta \mathscr{A} -\Phi'\frac{\partial^{2}\mathscr{A}}{\partial
{b}_{r}^{s}\partial
{b}_{t}^{u}}\nabla^{v}{b}_{r}^{s}\nabla_{v}{b}_{t}^{u}
+(\Phi'H-\Phi)\tr\left(\frac{\partial\mathscr{A}}{\partial
{b}_{i}^{j}}\right)-\Phi'\bigl|A\bigr|^{2}\frac{\partial\mathscr{A}}{\partial
{b}_{i}^{j}}{b}_{i}^{j}.
\end{align*}
The various terms on the  right hand side of this inequality can be easily
estimated: First, in view of Lemma \ref{approxiamtion f} $i)$ convexity of $\mathcal
{A}$ implies convexity of $\mathscr{A}$, then the second term can be
estimated by
\[
-\Phi'\frac{\partial^{2}\mathscr{A}}{\partial
{b}_{r}^{s}\partial
{b}_{t}^{u}}\nabla^{v}{b}_{r}^{s}\nabla_{v}{b}_{t}^{u}
\leq 0.
\]
Using Lemma \ref{approxiamtion f} $v)$,  the third
term can be estimated by
\[
(\Phi'H-\Phi).
\]
Lemma \ref{approxiamtion f} $iv)$ implies that the next term can be estimated by
\[
-\Phi'\bigl|{A}\bigr|^{2}(\mathscr{A}-(n-1)\delta).
\]
The following estimate is obtained:
\begin{align}\label{inverse weigargen evolution}
\partial_{t}\mathscr{A}
\leq \Phi'\Delta \mathscr{A}
+(\Phi'H-\Phi)
-\Phi'\bigl|{A}\bigr|^{2}(\mathscr{A}-(n-1)\delta).
\end{align}
{\bf Case 2.1. } for $\Phi'H-\Phi\leq 0$,
at a point $\left(p,t\right)$ with $\mathscr{A}-(n-1)\delta>0$,
this estimate \eqref{inverse weigargen evolution} gives the following estimate
\begin{align*}
\partial_{t}\mathscr{A}
\leq \Phi'\Delta \mathscr{A},
\end{align*}
which gives a contradiction if $\mathscr{A}$ attains a first maximum
larger than $(n-1)\delta$. The limit as $\delta$ is approached
to $0$ then implies the conclusion of the Lemma.
{\bf Case 2.2. } for $\Phi'H-\Phi>0$, at a point $\left(p,t\right)$ with $\mathscr{A}-(n-1)\frac{H\Phi'}{\Phi}\delta>0$,
\begin{align*}
\partial_{t}\mathscr{A}
&\leq \Phi'\Delta \mathscr{A}
+(\Phi'H-\Phi)
-\Phi'\bigl|{A}\bigr|^{2}(\mathscr{A}-(n-1)\delta)\\
&\leq \Phi'\Delta \mathscr{A}
+\left(\Phi'H-\Phi'\bigl|{A}\bigr|^{2}\mathscr{A}\right)
+\bigl|{A}\bigr|^{2}\frac{\mathscr{A}\Phi}{H}-\Phi\\
&= \Phi'\Delta \mathscr{A}
+\mathscr{A}\left(\Phi'-\frac{\Phi}{H}\right)\left(\frac{H}{\mathscr{A}}
-\bigl|{A}\bigr|^{2}\right).
\end{align*}
since\[\Phi'H-\Phi>0\] and
\[
\frac{H}{\mathscr{A}} \leq \frac{H}{\theta_{\max}}
=H{\lambda}_{\min}\leq \bigl|A\bigr|^{2},
\]
this gives
\begin{align*}
\partial_{t}\mathscr{A}
\leq \Phi'\Delta \mathscr{A},
\end{align*}
which gives a contradiction if $\mathscr{A}$ attains a first maximum
larger than $(n-1)\frac{H\Phi'}{\Phi}\delta$. The limit as $\delta$ is approached
to $0$ then implies the conclusion of the Lemma.
\end{proof}

\begin{corollary}\label{preserving weak convex}
Let $\mathrm{X}:M^{n}\times [0,T)\rightarrow
{\mathbb{R}}^{n+1}$ be a $\Phi(H)$-flow of strictly
 convex hypersurfaces. Then
\[
\bigl|A\bigr|\left(p,t\right)\leq H\left(p,t\right)
\leq
G^{-1}\left(G\left(\frac{1}{H_{\max}(0)}
\right)-t\right).
\]
\end{corollary}
\begin{proof}
Lemma \ref{preserving strict convex} implies that if $M_{0}$ is strictly  convex, under the
flow \eqref{eqation:1.1}, $M_{t}$ is strictly  convex as long as it
exists, then $\bigl|A\bigr|\leq H$, which implies that from the
evolution equation \eqref{the evolution for H} of $H$
\[
\partial_{t}{H}_{\max}\leq
\Phi H_{\max}^{2}.
\]
Now let $\phi$ be the solution of the ODE
\[
\left\{
\begin{array}{ll}
\frac{\dif\phi}{\dif t}=\Phi(\phi)\phi^{2},\\[2ex]
\phi(0)={H}_{\max}(0),
\end{array}\right.
\]
then by the maximum principle
\[
H\leq \phi \qquad \text{on} \qquad  0\leq t \leq T.
\]
On the other hand $\phi$ is explicitly given by
\[
\phi(t)=G^{-1}\left(G\left(\frac{1}{H_{\max}(0)}
\right)-t\right).
\]
Thus, this gives the desired estimate.
\end{proof}

\begin{corollary}\label{bound for maxH}
Let $\mathrm{X}:M^{n}\times [0,T)\rightarrow
{\mathbb{R}}^{n+1}$ be a $\Phi(H)$-flow of weakly
 convex hypersurfaces. Then $M_{t}$ is weakly  convex for all $t
\in [0, T)$ and $T_{\max}\geq G\left(\frac{1}{H_{\max}(0)}\right)$.
\end{corollary}
\begin{proof}
The initial surface $M_{0}$ can be smoothly approximated  by
strictly  convex hypersurfaces $M^{i}_{ 0}$, for example choosing the mean curvature flow.
Let these hypersurfaces move by $\Phi(H)$-flow,
 which by Lemma \ref{preserving strict convex} remain
strictly  convex.
By Theorem \ref{the long time existence for PhiH} and Corollary \ref{preserving weak convex}
we have a uniform lower bound $T_{\max}^{i}\geq G\left(\frac{1}{H_{\max}(0)}\right)$.
Using the uniform
$C^{2,\alpha}$-estimates from the proof of
Theorem \ref{the long time existence for PhiH}, one can extract a convergent subsequence of
strictly convex flows which implies the original flow also had
to be  convex.
\end{proof}

In the case that $\Phi''>0$, $ \Phi'\geq \frac{\Phi}{H}\geq 0$,
the following Proposition shows that weakly convex hypersurfaces immediately become strictly convex along the $\Phi(H)$-flow in ${\mathbb{R}}^{n+1}$
by using Lemma \ref{evolution for the quotient of Qr}.
\begin{proposition}\label{weakly convex become strong convex for Phi flow}
For $\Phi''>0$, $ \Phi'\geq \frac{\Phi}{H}$, let $M_{t}$ be a solution of the $\Phi(H)$-flow
in ${\mathbb{R}}^{n+1}$. Suppose the initial hypersurface
$M_{0}$ is a weakly convex hypersurface with $H_{\min}(0) > 0$. Then
$M_{t}$ is strictly  convex for all $t\in[0,T)$.
\end{proposition}
\begin{proof}
Since $H(t)\geq H_{\min}(0)>0$ for all  all $[0,T)$ along the $\Phi(H)$-flow,
${Q}_{2}$ is well-defined and  Corollary \ref{preserving weak convex} implies that
$M_{t}$ is weakly convex. Then an immediate consequence is
\[
{Q}_{2}=\frac{|{H}|^{2}-\bigl|{A}\bigr|^{2}}{2{H}}\geq
0.
\]
For $t\in [0,\varepsilon]$, $\varepsilon < T$, the bounds on
$\bigl|A\bigr|^{2}$ implies the bounds on
${Q}_{2}$ and $\Phi'$ which
implies
\[
\left[\Phi'\bigl|{A}\bigr|^{2}-r (\Phi' H-\Phi){Q}_{r}\right]{Q}_{2}
\leq C
\]
on this interval. An application of Lemma 3.4 for
$\omega:=\e^{Ct}{Q}_{2}$ shows the following estimate:
\[
\partial_{t}\omega \geq \Phi' \Delta \omega.
\]
Suppose that there exists $(p_{0},t_{0})\in
M^{n}\times(0,\varepsilon)$ with ${Q}_{2}(p_{0},t_{0})=0$,
then also $\omega(p_{0},t_{0})=0$. The Harnack's inequality in the
parabolic case (see i.e. \cite{Lie}) applied to the above equation
shows that $\omega\equiv 0$ for all $t\in (0,t_{0})$, i.e.
${Q}_{2}\equiv 0$, which is in contradiction to the existence
of strictly convex points on $M_{t}$, and so ${Q}_{2}>0$ on
$M^{n}\times(0,T)$. An iterative application of this yields that
${Q}_{r}>0$ on $M^{n}\times(0,T)$. This concludes the
Proposition.
\end{proof}

The following step want to show that the flow exists as long as it
bounds a non-vanishing volume. In order to achieve this, using a
trick of Tso \cite{Tso} for the $\mbox{Gau\ss}$ curvature flow, see
also \cite{And94}, \cite{Cab07} and \cite{McC}, study the evolution
under \eqref{eqation:1.1} of the following function
\begin{equation}\label{eqation:5.1}
Z_{t}=\frac{\Phi(H)}{\langle\mathrm{X},  \nu\rangle-\epsilon}.
\end{equation}
Here $\epsilon$ is a constant to be
chosen later.
\begin{corollary}\label{evoluton for Z}
For $t \in [0, T)$ and any constant $\epsilon$,
\begin{align*}
\partial_{t}Z=\Phi'\Delta Z
+\frac{2\Phi'}{\langle\mathrm{X},  \nu\rangle-\epsilon}\left\langle\nabla Z,\nabla\Phi\right\rangle +Z^{2} \left[\left(\frac{\Phi' H}{\Phi}+1\right)-\epsilon \frac{\Phi' \bigl|A\bigr|^{2}}{\Phi}\right].
\end{align*}
\end{corollary}
\begin{proof}

From \eqref{eqation:5.1}, \eqref{the evolution for H} and \eqref{evolution for support function}, it
follows
\begin{equation}\label{equation:5.10}
\begin{split}
\partial_{t} Z&=\frac{1}{\langle\mathrm{X},  \nu\rangle-\epsilon}
\left(\Phi'\Delta\Phi(H)+ \Phi'\Phi\left(\bigl|A\bigr|^{2}\right)\right)\\
&\quad-\frac{\Phi(H)}{(\langle\mathrm{X},  \nu\rangle-\epsilon)^{2}} \left(\Delta_{\dot \Phi} \langle\mathrm{X},  \nu\rangle
+ \bigl|A\bigr|^{2}\Phi'\langle\mathrm{X},  \nu\rangle- (\Phi' H+\Phi)\right).
\end{split}
\end{equation}
Another computation leads to
\begin{equation}\label{equation:5.11}
\begin{split}
\Phi'\Delta Z=\frac{\Phi'\Delta\Phi(H)}{\langle\mathrm{X},  \nu\rangle-\epsilon}
-\frac{ \Phi\Phi'\Delta\langle\mathrm{X},  \nu\rangle}{(\langle\mathrm{X},  \nu\rangle-\epsilon)^{2}} -2\frac{\Phi'}{\langle\mathrm{X},  \nu\rangle-\epsilon} \langle\nabla Z,\nabla \langle\mathrm{X},  \nu\rangle\rangle.
\end{split}
\end{equation}
Using \eqref{equation:5.11}, we can  simplify \eqref{equation:5.10} as the desired evolution equation for the function $Z$ easily.
\end{proof}

Now we apply the maximum principle to get an upper bound for $Z$
as long as the evolving hypersurface bounds a non-vanishing volume.
\begin{theorem}\label{upper bound for speed of Phi flow}
Let $M_{t}$ be a solution of the $\Phi(H)$-flow in
${\mathbb{R}}^{n+1}$, where the speed $\Phi(H)$ satisfies the conditions \ref{conditions on Phi}.
Suppose the initial hypersurface
$M_{0}$ is a  convex hypersurface, $\delta>0$, $q_{0}\in
{\mathbb{R}}^{n+1}$ and
$B_{\delta}(q_{0})\subset\Omega_{t}$ for all $t\in [0, \tau)$, which
boundary is $M_{t}$. Then
\[
H\left(p,t\right)\leq C(M_{0},\delta,n)\quad \text{for all}
\ \left(p,t\right)\in M^{n}\times [0,\tau).
\]
\end{theorem}
\begin{proof}
Without loss of generality, we take the point $q_{0}$ as the origin of ${\mathbb{R}}^{n+1}$
such that $\mathrm{X}$ is the position vector field.
Since it is proved previously that $M_{t}$ is  convex along the
flow \eqref{eqation:1.1}, there exists a constant $\epsilon>0$ in the definition \eqref{eqation:5.1} of
$Z$, $\epsilon=\epsilon(\delta)$,
such that the support function $\langle\mathrm{X}, \nu\rangle $
satisfies
 \[
\langle\mathrm{X}, \nu\rangle \geq 2\epsilon
\]
implies
\[
\langle\mathrm{X}, \nu\rangle-\epsilon\geq \epsilon>0.
\]
Combining this, convexity of $M_{t}$ implies that
$Z\geq 0$ and $\bigl|A\bigr|^{2}\geq\frac{1}{n}H^{2}$.
 From Corollary \ref{evoluton for Z}, the following
inequality can be obtained:
\begin{align*}
\partial_{t}Z&\leq \Phi'\Delta Z
+\frac{2\Phi'}{\langle\mathrm{X},  \nu\rangle-\epsilon}\left\langle\nabla Z,\nabla\Phi\right\rangle +Z^{2} \left[\left(\frac{\Phi' H}{\Phi}+1\right)-\epsilon \frac{\Phi' \bigl|H\bigr|^{2}}{n\Phi}\right]\\
&=\Phi'\Delta Z
+\frac{2\Phi'}{\langle\mathrm{X},  \nu\rangle-\epsilon}\left\langle\nabla Z,\nabla\Phi\right\rangle +Z^{2} \frac{\Phi' H}{n\Phi}\left[\left(\frac{\Phi }{H\Phi'}+1\right)-\epsilon \frac{ H}{n}\right].
\end{align*}
Notice that
\[
\left(\frac{\Phi }{H\Phi'}+1\right)'=-\frac{\Phi\Phi''H+\Phi \Phi'-\left(\Phi'\right)^{2}H }{\left(H\Phi'\right)^{2}}.
\]
From the assumptions \ref{conditions on Phi}, it follows that
\[
\left(\frac{\Phi }{H\Phi'}+1\right)'\leq 0,
\]
which implies the following estimate
\[
\left(\frac{\Phi}{H\Phi'}+1\right)\leq \left(\frac{\Phi(H_{\min}(0)) }{H_{\min}(0)\Phi'(H_{\min}(0))}+1\right)
:=D.
\]
Therefore, in view of the assumptions \ref{conditions on Phi} and the abouve estimate
 we bound the $\partial_{t}Z$ as follows
\begin{align*}
\partial_{t}Z\leq \Phi'\Delta Z
+\frac{2\Phi'}{\langle\mathrm{X},  \nu\rangle-\epsilon}\left\langle\nabla Z,\nabla\Phi\right\rangle +Z^{2} \frac{\Phi' H}{n\Phi}
\left(D-\epsilon \frac{ H}{n}\right).
\end{align*}

Assume that in $(p_{0},t_{0})$, $Z$ attains a big maximum $C\gg 0$ for
the first time. Then
\[
\Phi(H)(p_{0},t_{0})\geq C(\langle\mathrm{X},  \nu\rangle-\epsilon)(p_{0},t_{0})\geq
\epsilon C,
\]
which gives a contradiction if
\[
C\geq \max_{p\in
M^{n}}\left\{Z(p,0),\frac{1}{\epsilon}\Phi\left(\frac{n D}{\epsilon}\right)\right\}.
\]
\end{proof}

\begin{proof}[Proof of Theorem \ref{main result for Phi flow}]
Theorem \ref{the long time existence for PhiH} and Theorem \ref{upper bound for speed of Phi flow}
 ensure that the $\Phi(H)$-flow exists as long as it bounds a non-vanishing domain.
 Lemma \ref{graph represent for PhiH} and Proposition \ref{weakly convex become strong convex for Phi flow} show
that all hypersurfaces are strictly  convex for  $t\in [0, \tau)$,
thus $\lim_{t\rightarrow T}\lambda_{\min}(t)\geq \delta>0$.
Now by adapting Schulze's approach in the case $H^{\beta}$-flow in \cite{Sch05},
one can complete the remainder of the proof for convergence to a single point as the final time is
approached.
\end{proof}

\bigskip
\noindent{\bf  Acknowledgments.}
\quad  This paper was written while I was a post-doctor fellow
at School of Mathematics, Sichuan University.
I am grateful to the School for providing an ideal working atmosphere.
I would like to express special thanks to Professor An-Min Li for his enthusiasm and encouragement, and
Professor Guanghan Li for his patience, suggestion and helpful discussion on this topic. The research is partially supported by China Postdoctoral Science Foundation Grant 2015M582546
and Natural Science Foundation Grant 2016J01672 of the Fujian Province, China.

\newcommand{\J}[4]{{\sl #1} {\bf #2} (#3) #4}
\newcommand{\CMP}{Comm.\ Math.\ Phys.}

\end{document}